\documentclass{article}
\usepackage[utf8]{inputenc}
\usepackage{amsmath}
\usepackage{amssymb}
\usepackage{amsfonts}
\usepackage{amsthm}
\usepackage{mathtools}
\usepackage{enumerate}
\usepackage[hidelinks]{hyperref}
\usepackage{titlesec}

\usepackage[shortlabels]{enumitem}
\usepackage{yfonts}
\usepackage{setspace}
\usepackage[utf8]{inputenc}
\usepackage[english]{babel}
\usepackage[margin = 1in]{geometry}
\usepackage{tikz}
\usepackage{graphicx}
\usepackage{authblk}
\usetikzlibrary{calc}
\usepackage{fancyhdr}
\usepackage{float}
\DeclarePairedDelimiter{\ceil}{\lceil}{\rceil}
\DeclarePairedDelimiter\floor{\lfloor}{\rfloor}

\titleformat{\section}[block]{\centering\Large}{\thesection}{1em}{}

\theoremstyle{definition}
\newtheorem*{defi}{Definition}

\theoremstyle{remark}
\newtheorem*{rem}{Remark}
\theoremstyle{plain}

\newtheorem{que}{Question}[section]
\newtheorem{thm}{Theorem}[section]
\newtheorem{prop}[thm]{Proposition}
\newtheorem{lem}[thm]{Lemma}
\newtheorem{cor}[thm]{Corollary}

\title{Classification of Lattices Bounded by Large Surgeries of Knots}

\author{Ali Naseri Sadr}
\date{}

\begin{document}

\maketitle
\begin{abstract}
    We classify all the lattices realized as the intersection form of a positive definite four manifold with boundary $S_n^3(K)$ for a knot $K$ in the three sphere and a positive integer $n$ greater than $4g_4(K)+3$. We then use this result to define a concordance invariant and generalize a theorem of Rasmussen on lens space surgeries.
\end{abstract}
\section{Introduction}
Inspired by the results in \cite{golla2019definite}, we give a complete classification of the lattices arising as the intersection form of a positive definite four manifold with boundary $X$, where the boundary $Y$ is large surgery along a knot $K$ in $S^3$. We then use this result to define a concordance invariant $l(K)$ and examine how it behaves under crossing change. In the last section, we generalize a result of Rasmussen from \cite{Rasmussen_2004} on lens space surgeries using our main theorem. 
\begin{defi}
\label{Lattice Def}
    We say an oriented three manifold $Y$ bounds a lattice $L$ if there exists an oriented four manifold $X$ with no torsion in its homology such that $\partial X=Y$ as an oriented manifold and $Q_X$ is isomorphic to $L$. 
\end{defi}
Our first goal is to prove the following theorem.
\begin{thm}
\label{Main THM}
    Consider a knot $K$ in $S^3$ and let $n$ be an integer greater than $4g_4(K)+3$ where $g_4(K)$ is the slice genus of $K$. Suppose $S^3_n(K)$ bounds a lattice $L$. Then $L$ is isomorphic to 
    $\langle n\rangle\oplus\langle 1\rangle^{rk(L)-1}$.
\end{thm}

    We note that if $K$ is slice, then we get a contractible four manifold bounding $S^3_{1}(K)$ by surgery along a slice disk for $K$. This shows $S^3_1(K)$ can only bound the Euclidean lattice by Donaldson's theorem. There is a positive definite 2-handle cobordism from $S^3_n(K)$ to $S^3_1(K)$ when $n$ is a positive number. Using this cobordism and Donaldson's theorem, one can prove the previous theorem for slice knots and positive integers by induction on $n$. This heuristic shows if one can classify all the lattices bounded by $S^3_1(K)$, then one can try to classify all the lattices bounded by $S^3_n(K)$ inductively. Indeed, this is the method used in \cite{golla2019definite}
    . However, we use a classification result for non-unimodular lattices similar to the one proved by Elkies in \cite{elkies1999characterization} for the unimodular ones and the proof follows from correction terms in Heegaard Floer homology; our proof is similar to the proof of Donaldson's theorem in \cite{ozsvath2002absolutely}
    by Ozsvath and Szabo.

Now we can use Theorem \ref{Main THM} to define a concordance invariant $l(K)$ for every knot $K$ in $S^3$.
\begin{defi}
    Consider a knot $K$ and let $n$ be a positive integer. We say $S_n^3(K)$ bounds a non-standard lattice if it bounds a lattice $L$ that is not isomorphic to $\langle n \rangle\oplus\langle 1\rangle^{{rk(L)}-1}$.
 \end{defi}
 \begin{defi}
     Let $K$ be a knot in $S^3$. Define
     \begin{equation}
     l(K)\coloneqq\sup\hspace{0.5mm}\{ n : S^3_n(K)\hspace{0.75mm} \text{bounds a non-standard lattice}\}.
     \end{equation}
 \end{defi}
 According to Theorem \ref{Main THM} this is a finite number less than or equal to $4g_4(K)+3$ and it vanishes for slice knots. Let $K$ be a non-trivial $L$-space knot and assume that the $L$-space surgery slopes are negative which can be achieved by mirroring the knot if necessary. Using our main theorem, we prove

 \begin{thm}
 \label{Slope Bounds}
     Suppose $S_{-n}^3(K)$ admits a sharp negative definite filling for some positive integer $n$. Then
     \begin{equation}
         n\leq l(m(K)),
     \end{equation}
     where $m(K)$ denotes the mirror of $K$; in particular, we get $n\leq 4g(K)+3$.
 \end{thm}
 \begin{rem}
     This generalizes the main theorem in \cite{Rasmussen_2004} since if we assume $S^3_{-n}(K)$ is a lens space, then it has a sharp negative definite filling and our theorem implies $n$ must be less than or equal to $4g(K)+3$.
 \end{rem}
 \begin{rem}
     Note that $S^3_{-n}(K)$ with the reversed orientation is the same as $S^3_n(m(K))$; this three manifold might have a sharp negative definite filling for arbitrary large $n$ and the theorem does not hold for positive surgeries; see Theorem $1.2$ in \cite{McCoy_2021}.
 \end{rem}
 
\section*{Acknowledgments}
The author is grateful to his advisors, John Baldwin and Josh Greene, for their invaluable guidance, support, and insightful conversations about this work.
He would also like to express his gratitude to Christopher Scaduto for an informative discussion on this work. 
\section{Surgery Formula for Correction Terms}
We assume the reader is familiar with Heegaard Floer homology and Knot Floer homology as explained in \cite{rasmussen2003floer} and \cite{ozsvath2003holomorphic}.
Let $K$ be a knot in $S^3$ and $CFK^{\infty}(K)$ denote its $\mathbb{Z}\oplus\mathbb{Z}$ filtered knot complex. There are two filtrations on this complex that we denote by $i$ and $j$. Let $B^{+}$ be the quotient complex corresponding to all the elements with $i\geq 0$ and let $A_s$ denote the quotient complex corresponding to all the elements with $\max(i,j-s)\geq 0$. The complex $B^+$ is chain homotopic to $CF^{+}(S^3)$ and the large surgery formula realizes each $A_s$ as $CF^+(Y,\mathfrak{t})$ where $Y$ is a large surgery along $K$ and $\mathfrak{t}$ is a Spin$^{\mathbb{C}}$ structure on $Y$. In particular, we have
$$
H(B^+)\cong \mathcal{T}^+_{0}
$$
where $\mathcal{T}^+$ is the $\mathbb{F}[U]$-module $\mathbb{F}[U,U^{-1}]/U\cdot\mathbb{F}[U]$ and $\mathcal{T}^+_{0}$ means $1$ is supported in grading $0$. We also have
$$
H(A_s)\cong \mathcal{T}^{+}\oplus M
$$
where $M$ is a $\mathbb{F}[U]$-torsion module.
There are natural chain maps $v_s\colon A_s\to B^
+$ defined by mapping the generators with $i<0$ in $A_s$ to $0$. The induced maps in homology take the tower part of $H(A_s)$ to the tower part of $H(B^+)$. Since this is a $U$-equivariant map from a tower to another one, it has to be multiplication by a power of $U$; we denote this power by $V_s$. In \cite{rasmussen2003floer}, Rasmussen proved for each $K$ and $i$, we have
$$
V_i(K)-1\leq V_{i+1}(K)\leq V_i(K).
$$
Rasmussen also proved
\begin{equation}
\label{v i bound}
    V_i(K)\leq \ceil[\bigg] {\frac{g_4(K)-i}{2}}
\end{equation}
for $0 \leq i< g_4(K)$ and $V_i(K)=0$ for $i\geq g_4(K)$.
Let $n$ be a positive integer and consider the natural $2$-handle cobordism from $S^3$ to $S^3_n(K)$; we denote this cobordism by $X_n(K)$. Fix a Seifert surface $\Sigma$ for $K$ and cap it off with a disk in $X_n(K)$. We call the resulting closed surface $\Tilde{\Sigma}$ and define a map $\rho\colon\text{Spin}^{\mathbb{C}}(S^3_n(K))\to \mathbb{Z}/n\mathbb{Z}$ using this closed surface. Consider a Spin$^{\mathbb{C}}$ structure $\mathfrak{t}$ on $S^3_n(K)$ and extend it to a Spin$^{\mathbb{C}}$ structure $\mathfrak{s}$ on $X_n(K)$. We have
$$
\langle c_1(\mathfrak{s}),\Tilde{\Sigma}\rangle-n \equiv 2i \hspace{2mm} \mod 2n.
$$
Define $\rho(\mathfrak{t})\coloneqq i$. One can show this is independent of the choice of $\mathfrak{s}$ and the Seifert surface $\Sigma$; it is also a bijection and we will use $\rho$ to identify Spin$^{\mathbb{C}}(S^3_n
(K))$ with $\mathbb{Z}/n\mathbb{Z}$. The following was proved by Ni and Wu in \cite{ni2013cosmetic}.
\begin{prop}
\label{d invt formula}
    Suppose $K$ is a knot in $S^3$ and $p,q$ are positive numbers. Then for $0\leq i\leq p-1$, we have
    \begin{equation}
    d(S_{\frac{p}{q}}^3(K),i)=d(L(p,q),i)-2\max(V_{\floor{\frac{i}{q}}},V_{\floor{\frac{p+q-1-i}{q}}}),
    \end{equation}
    where $L(p,q)$ is the $\frac{p}{q}$ surgery on the unknot.
\end{prop}
\begin{rem}
    The affine identification between Spin$^{\mathbb{C}}(S^3_{\frac{p}{q}}(K))$ and $\mathbb{Z}/p\mathbb{Z}$ is slightly different from the one given for integer surgeries when $q$ is greater than one, but we only need the case of integer surgeries in this note.
\end{rem}
\section{A Characterisation of Non-Unimodular Definite Lattices}
Consider a lattice $L$ and let $Q$ denote the pairing on $L$. We can extend this paring to $L\otimes \mathbb{Q}$ by
$$
Q^*(x\otimes p, y\otimes q)\coloneqq pqQ(x,y).
$$
Define the dual lattice $L^*\subset L\otimes \mathbb{Q}$ by
$$
L^*=\{x\otimes p: Q^*(x\otimes p,y)\in \mathbb{Z}\hspace{2mm}\forall y\in L\}.
$$
There is a natural inclusion from $L$ to its dual and we call $L^*/L$ the discriminant group of $L$. We define
$det(L)$ to be  $|L^*/L|$.
An element $\xi$ in $L^*$ is called a characteristic covector if 
$$
Q(x,x)\equiv Q(x,\xi)\hspace{2mm} \mod 2
$$
for every $x$ in $L$. We denote the set of characteristic covectors by $char(L)$.
The following was proved by Owens and Strle in \cite{owens2008characterisation}; it is a generalization of Elkies's theorem for unimodular lattices in \cite{elkies1999characterization}.
\begin{thm}
\label{lattice char}
    Let $L$ be a positive definite lattice of rank $r$ and determinant $\delta$. Then there exists a characteristic covector $\xi$ in $L^*$ with 
    \begin{equation}
    Q^*(\xi,\xi)\leq 
       \begin{cases}
            r-1+\frac{1}{\delta}\hspace{3mm} \text{if $\delta$ is odd}&\\
            r-1 \hspace{7mm} \text{if $\delta$ is even;}
       \end{cases} 
    \end{equation}
    this inequality is strict unless $L\cong (r-1)\langle 1\rangle\oplus \langle \delta\rangle$. Moreover, the two sides of the inequality are congruent modulo $\frac{4}{\delta}$. 
\end{thm}
\begin{rem}
    This theorem implies if $L$ is a lattice of rank $r$ and 
odd determinant $\delta$ with
    $$
    r-1+\frac{1}{\delta}\leq \min_{\xi\in char(L)}Q^*(\xi,\xi),
    $$
    then $L$ is isomorphic to $(r-1)\langle 1\rangle\oplus \langle \delta\rangle$. The analogous result holds for lattices with even determinant with $r-1$ in place of $r-1+\frac{1}{\delta}$.
\end{rem}
\section{Classification of Lattices Bounded by Large Surgeries}
Before stating the proof of Theorem \ref{Main THM}, we need some preliminaries. Let $X$ be a four-manifold with boundary a rational homology sphere $Y$. Writing the long exact sequence for singular homology with integer coefficients, we get
\begin{equation}
\label{lES}
0\rightarrow H_2(X)\rightarrow H_2(X,Y)\rightarrow H_1(Y)\rightarrow H_1(X).
\end{equation}
If we assume the homology of $X$ has no torsion, then $Q_X$ is defined on $H_2(X)\times H_2(X)$ and Poincare-Lefschetz duality proves $H_2(X,Y)$ is isomorphic to the dual of this lattice; since there is no torsion in homology of $X$, the last arrow in the long exact sequence is zero and this proves 
$$
H_1(Y)\cong L^*/L
$$
where $L$ denotes $H_2(X)$. In particular, we get $|H_1(Y)|=det(L)$.
\\
The following was proved by Ozsvath and Szabo in \cite{ozsvath2002absolutely}.
\begin{thm}
\label{d-invt ineq}
    Suppose $X$ is a compact oriented positive definite four-manifold with boundary a rational homology sphere $Y$, and $\mathfrak{s}$ is a Spin$^\mathbb{C}$ structure on $X$. Then
    \begin{equation}
        -4d(Y,\mathfrak{t})\geq b_2(X)-c_1(\mathfrak{s})^2,
    \end{equation}
    where $\mathfrak{t}$ is the restriction of $\mathfrak{s}$ to $Y$ and $c_1(\mathfrak{s})$ denotes the first Chern class of $\mathfrak{s}$.
\end{thm}
\begin{rem}
    If $\mathfrak{s}$ is a Spin$^\mathbb{C}$ structure on $X$, then $c_1(\mathfrak{s})$ is in $H^2(X)\cong H_2(X,Y)$ and the mod $2$ reduction of $c_1(\mathfrak{s})$ is the second Stiefel Whitney class of $X$. Hence, $c_1(\mathfrak{s})$ is a characteristic covector of the intersection form of $X$ and $c_1(\mathfrak{s})^2$ denotes $Q_X^*(c_1(\mathfrak{s}),c_1(\mathfrak{s}))$.
\end{rem}
By Proposition \ref{d invt formula}, we can write
$$
d(S^3_n(K),i)=d(L(n,1),i)-2\max(V_i,V_{n-i})
$$
for every knot $K$ in $S^3$ and positive integer $n$. The correction terms for $L(n,1)$ are given by
$$
d(L(n,1),i)=\frac{(2i-n)^2-n}{4n}
$$
for $0\leq i\leq n-1$.
Thus we get
$$
d(S^3_n(K),i)=\frac{(2i-n)^2-n}{4n}-2\max(V_i,V_{n-i}).
$$
If $g_4(K)\leq\min(i,n-i)$, then both $V_i$ and $V_
{n-i}$ are zero by equation \eqref{v i bound} and we get
$$
d(S^3_n(K),i)=\frac{(2i-n)^2-n}{4n}.
$$
Now define
\begin{equation*}
\beta(n)\coloneqq
    \begin{cases}
        \frac{1}{n} -1 &\hspace{2mm} n=1 \hspace{2mm} \mod 2,\\
        -1 &\hspace{2mm} n=0 \hspace{2mm} \mod 2.
    \end{cases}
\end{equation*}
We conclude that 
$$
\beta(n)\leq 4d(S^3_n(K),i)
$$
for every $i$ with $g_4(K)\leq\min(i,n-i)$ and every knot $K$ in $S^3$.
\begin{lem}
\label{bounding lemma}
    Suppose $K$ is a knot in $S^3$ and $n$ is a positive integer greater than $4g_4(K)+3$. Then 
    \begin{equation}
        \label{beta ineq}
        \beta(n)\leq 4d(S^3_n(K),i)
    \end{equation}
    for every $0\leq i\leq n-1$.
\end{lem}
\begin{proof}
    We know the inequality holds for $i$ with $g_4(K)\leq\min(i,n-i)$; assume $g_4(K)>\min(i,n-i)$. Since
    $$
    d(S^3_n(K),i)=d(S^3_n(K),n-i),
    $$
    without loss of generality, we can assume $0\leq i<g_4(K)$. Using Proposition \ref{d invt formula} and equation \eqref{v i bound}, we have
    \begin{align*}
        &4d(S^3_n(K),i)=\frac{(2i-n)^2-n}{n}-8\max(V_i,V_{n-i})=\\
        &\frac{(2i-n)^2-n}{n}-8V_i\geq \frac{(2i-n)^2-n}{n}-8\big(\frac{g_4-i+1}{2}\big )\\
        &=-1+\frac{4i^2}{n}+(n-4g_4(K)-4)\geq \beta(n).
    \end{align*}
\end{proof}
\begin{proof}[\textbf{Proof of Theorem \ref{Main THM}}]
Let $K$ be a knot in $S^3$ and consider an integer $n$ greater than $4g_4(K)+3$. Suppose $S^3_n(K)$ bounds a four-manifold $X$ with no torsion in its homology and intersection form given by a lattice $L$. By the long exact sequence in \eqref{lES}, we get
$$
det(L)=n.
$$
Moreover, every $\xi$ in $char(L)$ corresponds to a spin$^\mathbb{C}$ structure $\mathfrak{s}$ on $X$ and we have
\begin{equation*}
    -4d(S^3_n(K),\mathfrak{t})\geq rk(L)-Q_X^*(\xi,\xi)
\end{equation*}
by Theorem \ref{d-invt ineq}. Rewriting this inequality and using equation \eqref{beta ineq}, we get
\begin{equation*}
    Q_X^*(\xi,\xi)\geq rk(L)+\beta(det(L))
\end{equation*}
for every characteristic covector of $L$.
We conclude $L$ is isomorphic to $\langle 1\rangle ^{rk(L)-1}\oplus \langle n \rangle$ by Theorem \ref{lattice char}.
\end{proof}
\begin{rem}
    If one assumes $X$ has torsion in its homology, but $n$ is square free, then the last arrow in \eqref{lES} vanishes and $H_1(Y)$ becomes isomorphic to the discriminant group of the lattice $Q_X$ where the intersection form is defined on $H_2(X)/Tor$. Therefore, we get the following corollary. 
\end{rem}
\begin{cor}
    Let $K$ be a knot in $S^3$ and $n$ a square free integer greater than $4g_4(K)+3$. Suppose a positive definite four-manifold $X$ bounds $S^3_n(K)$. Then
    $$
    Q_X\cong \langle 1\rangle^{b_2(X)-1}\oplus\langle n\rangle.
    $$
\end{cor}
\begin{rem}
    Consider the torus knot $T(2,n)$ where $n$ is odd and greater than $1$. It is shown in \cite{Moser1971ElementarySA}
    that
    $$
    S^3_{2n+1}(T(2,n)) \cong L(2n+1,4)
    $$
    This lens space bounds the linear lattice $\Lambda(2n+1,4)$. There is no element with self intersection one in this lattice if $n$ is greater than one. Hence, it cannot be isomorphic to $\langle 2n+1 \rangle\oplus\langle 1\rangle$. Note that we have
    $$
    2n+1 = 4g_4(T(2,n))+3.
    $$
    In particular, this shows the bound in Theorem \ref{Main THM} is sharp.
\end{rem}
 \section{A Concordance Invariant}
 In this section, we investigate the invariant $l(K)$ defined in the introduction.
 \begin{prop}
     Let $K$ be a knot in $S^3$ and assume $V_0(K)$ is zero. Then $l(K)$ is zero.
 \end{prop}
 \begin{proof}
     If $V_0(K)$ is zero, then 
     \begin{equation*}
         \beta(n)\leq 4d(S^3_n(K),i)
     \end{equation*}
     for every positive integer $n$ and $0 \leq i\leq n-1$. Hence, the result follows from the proof of Theorem \ref{Main THM}.
 \end{proof}
 \begin{prop}
Let $K_1$ and $K_2$ be two concordant knots in $S^3$. Then $l(K_1)=l(K_2)$. 
 \end{prop}
 \begin{proof}
     Fix a positive integer $n$ and let $Y_1$ and $Y_2$ denote $S^3_n(K_1)$ and $S^3_n(K_2)$ respectively. Consider a properly embedded annulus $A$ in $S^3\times I$ going from $K_1$ to $K_2$. We can extend $n$ surgery along $K_1$ and $K_2$ to $A$ and the resulting four manifold will be a homology cobordism from $Y_1$ to $Y_2$. Denote this homology cobordism by $W$. Let $L$ be a positive definite lattice bounded by $Y_1$ and $X_1$ be the four manifold with intersection form $L$ and $\partial X_1 = Y_1$. We glue $X_1$ to $W$ along $Y_1$ and call the resulting four manifold $X_2$. This four manifold has the same intersection form and homology as $X_1$ since $W$ is a homology cobordism. Hence, the lattice $L$ is also bounded by $Y_2$. In particular, the sets of lattices bounded by $Y_1$ and $Y_2$ are the same; we conclude that $l(K_1)=l(K_2)$ 
 \end{proof}
 Let $T$ be a null-homologous knot in a three manifold $Y$ and fix a positive integer $n$. Denote the $\frac{1}{n}$-surgery on $Y$ along $T$ by $Y_{\frac{1}{n}}(T)$. We can write $\frac{1}{n}$ as a continued fraction given by
 $$
 [1,2,\dots,2]^{-},
 $$
 where there are $n-1$, $2$'s in the continued fraction. We can use this to construct a two handle cobordism from $Y$ to $Y_{\frac{1}{n}}(T)$. We denote this cobordism by $X_{\frac{1}{n}}(T)$; this four manifold is positive definite and its intersection form is isomorphic to $\langle 1\rangle^{n}$.
 \begin{defi}
     We say two definite lattices $L_1$ and $L_2$ are stably equivalent if there exist non-negative integers $b_1$ and $b_2$ such that
     $$
     L_1\oplus\langle 1\rangle ^{b_1}\cong L_2 \oplus \langle 1 \rangle^{b_2}.
     $$
     This is an equivalence relation among positive definite lattices. 
 \end{defi}
  \begin{defi}
     Fix a knot $K$ in $S^3$ and let $n$ be a positive integer. We define $L(K,n)$ to be the set of all positive definite lattices that bound $S^3_n(K)$ up to stable equivalence. 
 \end{defi}
 \begin{lem}
  Let $K$ be a knot in $S^3$ and $c$ a negative crossing of $K$. Consider a crossing disk $D$ for $c$ and denote its boundary by $T$. Fix a positive integer $m$ and let $K_m$ denote the knot obtained from $K$ by performing $\frac{1}{m}$ surgery along $T$. We have
  $$
  L(K,n)\subseteq L(K_m,n)
  $$
  for every positive integer $n$.
 \end{lem}
 \begin{proof}
     Let $Y_1$ and $Y_2$ denote $n$ surgery on $K$ and $K_m$ respectively. The knot $T$ is a null-homologous knot in $Y_1$ since it has zero linking number with $K$. Consider the two handle cobordism $X_{\frac{1}{m}}(T)$ from $Y_1$ to $Y_2$. Suppose $Y_1$ bounds a lattice $L$, $X_1$ is a four manifold with intersection form $L$ and $\partial X_1=Y_1$. We glue $X_1$ to $X_{\frac{1}{m}}(T)$ along $Y_1$ and denote the resulting four manifold by $X_2$. Since $T$ is a null-homologous knot in $Y_1$, we get
     $$
     Q_{X_2} \cong Q_{X_1}\oplus\langle 1\rangle^{m}\cong L\oplus\langle 1\rangle ^m.
     $$
     The four manifold $X_2$ does not have torsion in its homology because $X_1$ does not have torsion in its homology by assumption and $X_{\frac{1}{m}}(T)$ is a two handle cobordism. Hence, $Y_2$ bounds $L\oplus\langle 1\rangle^m$ and we conclude the claim.
 \end{proof}
 \begin{cor}
     For every positive integer $m$, we have
     $$
     l(K)\leq l(K_m).
     $$
     In particular, if $K^+$ denote the knot obtained from another knot $K^-$ by changing a negative crossing $c$ to a positive one, then $l(K^-)\leq l(K^+)$.
 \end{cor}
 \begin{defi}
     A knot $K$ in $S^3$ is called negative if it admits a diagram without positive crossings.
 \end{defi}
 \begin{cor}
     Let $K$ be a negative knot. Then $l(K)=0$ and $S_n^3(K)$ does not bound any non-standard lattice. 
 \end{cor}
 We conclude this section with a remark about the invariant $l(K)$. Consider a knot $K$ with $l(K)>1$ and let $n$ be a positive integer less than $l(K)$. The definition for $l(K)$ does not imply that $S^3_n(K)$ bounds a non-standard lattice, but this is in fact true. If we consider the natural positive definite two handle cobordism from $S^3_{l(K)}(K)$ to $S^3_n(K)$ and glue it to the non-standard filling of $S^3_{l(K)}(K)$, we get a positive definite filling of $S^3_n(K)$ with no torsion in its homology; it remains to prove the intersection form of this filling is also non-standard and this follows from Lemma $3.1$ in \cite{golla2019definite} and the inductive argument we mentioned in the introduction. In other words, if the filling for $S^3_n(K)$ was standard, then the filling for $S^3_{l(K)}(K)$ would be standard which contradicts the definition of $l(K)$. Hence, we get
 \begin{prop}
     Let $K$ be a knot in the three sphere with $l(K)$ greater than zero. Then for every positive integer $n$ less than or equal to $l(K)$, the three manifold $S^3_n(K)$ bounds a non-standard lattice.
 \end{prop}
      
 \section{Lens Space Surgeries and Sharp Fillings}
 In \cite{Rasmussen_2004}, Rasmussen proved if a knot $K$ admits an integer lens space surgery, then the absolute value of the surgery slope is less than or equal to $4g(K)+3$. In this section, we generalize this result to surgery of $L$-space knots bounding sharp fillings.
 Let $Y$ be a $L$-space and $X$ a negative definite four manifold without torsion in its homology bounding $Y$.
 \begin{defi}
     We say $X$ is a sharp filling of $Y$ if for every $\mathfrak{t}$ in spin$^{\mathbb{C}}(Y)$, there is a spin$^\mathbb{C}$ structure $\mathfrak{s}$ on $X$ that restricts to $\mathfrak{t}$ on $Y$ and we have
     \begin{equation}
     \label{sharp eq}
         d(Y,\mathfrak{t}) = \frac{c_1(\mathfrak{s})^2+b_2(X)}{4}.
     \end{equation}
     This is equivalent to $\widehat{HF}(X\setminus B^4,\mathfrak{s})$ being an isomorphism from $\widehat{HF}(S^3)$ to $\widehat{HF}(Y,\mathfrak{t})$.  
 \end{defi}
 \begin{rem}
     For instance, the linear plumbing $-X(p,q)$ is a sharp filling for $L(p,-q)$.
 \end{rem}
 \begin{lem}
 \label{int form of sharp fillings}
     Let $K$ be a $L$-space knot and assume $S^3_{-n}(K)$ is a $L$-space which admits a sharp negative definite filling $X$ for some positive integer $n$. If we have
     $$
     Q_X\cong \langle -n\rangle\oplus\langle -1\rangle^{b_2(X)-1},
     $$
     then $K$ is the unknot.
 \end{lem}
 \begin{proof}
     Combining the long exact sequence in \eqref{lES} with the isomorphism between $Q_X$ and $\langle -n\rangle\oplus\langle -1\rangle^{b_2(X)-1}$, we can find an affine isomorphism $\sigma$ between spin$^\mathbb{C}(L(n,-1))$ and spin$^\mathbb{C}(S^3_{-n}(K))$ such that
     $$
     d(L(n,-1),i)=d(S^3_{-n}(K),\sigma(i))
     $$
     for every $i$ because the filling $X$ is sharp and has the same intersection form as $-X(n,1)$ up to stabilization. Hence, we get
     $$
     \lambda(L(n,-1))=\sum_{i=0}^{n-1} d(L(n,-1),i)=\sum_{i=0}^{n-1} d(S^3_{-n}(K),i)=\lambda(S^3_{-n}(K)),
     $$
     where $\lambda$ denotes the Casson-Walker invariant. Using the surgery formula for Casson-Walker invariant, we conclude
     $$
     \Delta_K^{''}(1)= \frac{1}{n}\cdot(\lambda(S^3_{-n}(K))-\lambda(L(-n,1)))=0.
     $$
     The only $L$-space knot with vanishing $\Delta_K^{''}(1)$ is the unknot; see the proof of Theorem $1.4$ in \cite{5} for more details.
 \end{proof}
 \begin{proof}[\textbf{Proof of Theorem \ref{Slope Bounds}}]
     Since $K$ is non-trivial, Lemma \ref{int form of sharp fillings} implies that $S^3_{-n}(K)$ has a sharp filling $X$ where $Q_X$ is not isomorphic to 
     $$
     \langle -n\rangle\oplus\langle -1\rangle^{b_2(X)-1}.
     $$
     Now consider the mirror of $K$; the four manifold $-X$ is a positive definite filling for $n$ surgery on the mirror of $K$ and this filling is non-standard in the sense of previous section. Therefore, we must have
     \begin{equation*}
         n\leq l(m(K)) \leq4g(K)+3
     \end{equation*}
     by Theorem \ref{Main THM}.
 \end{proof}
 \begin{cor}
     Suppose $K$ is a non-trivial knot that admits an integer lens space surgery. Then the absolute value of the surgery slope is less than or equal to $4g(K)+3$.
 \end{cor}
 \begin{proof}
     If necessary, we can mirror the knot $K$ so that $S_{-n}^3(K)$ is a lens space; every lens space bounds a sharp negative definite filling and the claim follows from Theorem \ref{Slope Bounds}. 
 \end{proof}
 \begin{rem}
     If a non-trivial knot $K$ admits an integer lens space surgery $L(p,q)$, then $\Lambda(p,q)$ embeds as a changemaker lattice in the Euclidean lattice with codimension one and we can find $g(K)$ in terms of the changemaker coordinates; this was proved by Greene in \cite{greene2012lspace}. McCoy used this result and gave another proof of corollary $6.2$ in \cite{McCoy_2017}   
 \end{rem}
 \begin{rem}
     If $S_{4g(K)+3}^3(K)$ is a lens space, then it is possible to prove $K$ is in fact an alternating torus knot; see \cite{Rasmussen_2004} for more details. 
 \end{rem}
 \section{Conclusion}
 We conclude this note with some questions and speculations. There are two ways of generalizing Theorem \ref{Main THM}; the first is one to ask whether such a result would hold for rational surgeries along a knot $K$ with slopes greater than 
a positive number $N(K)$ depending on $K$. Assume $r=\frac{p}{q}$ is a positive rational number that is large enough in comparison to $g_4(K)$. The three manifold $S_{r}^3(K)$ bounds a positive definite two handle cobordism $X$ with $Q_X\cong\Lambda(p,q)$.
 \begin{que}
     Suppose $X^{'}$ is a positive definite four manifold with no torsion in its homology and $\partial X^{'} = S^3_r(K)$ as an oriented manifold. Can one prove that
     \begin{equation*}
         Q_{X^{'}}\cong\Lambda(p,q)\oplus\langle1\rangle^b
     \end{equation*}
     for some non-negative integer $b$?
 \end{que}
 In order to answer this question in affirmative, one would need a characterization result for $\Lambda(p,q)$ similar to the one given in \cite{owens2008characterisation} for $\langle n\rangle\oplus\langle 1\rangle^b$. The second way to generalize theorem \ref{Main THM} is finding a similar result for integer surgeries along components of a link. Let $L$ be a link in $S^3$ with $h$ components and consider a vector $v=(v_1,v_2,\dots,v_h)$ in $\mathbb{Z}_{> 0}^h$. Denote the three manifold obtained from performing $v_i$-surgery along the $i$-th component of $L$ by $S^3_v(L)$ and let $X$ be the trace of this surgery. If we have
 \begin{equation*}
     \sum_{i=1}^{i=h}v_i>\sum_{K_i\neq K_j\in L} lk(K_i,K_j),
 \end{equation*}
 then $X$ is a positive definite four manifold.
 \begin{que}
     Fix a link $L$ in $S^3$. Does there exist a positive integer $N(L)$ such that if $v_i$ is greater than $N(L)$ for every $i$, then every positive definite four manifold $X^{'}$ bounding $S^3_{v}(L)$ satisfies
     \begin{equation*}
         Q_{X^{'}}\cong Q_X\oplus\langle1\rangle^b
     \end{equation*}
     for some non-negative integer $b$?
 \end{que}
 Our last two questions are about the behaviour of $l(K)$ under negative crossing changes and connected sums.
 \begin{que}
     Let $K^+$ be a knot with a positive crossing and $K^-$ denote the knot resulting from changing the crossing. Is there a fixed positive integer $N$ such that $l(K^+)\leq l(K^-)+N$ for every $K^+$?
 \end{que}
 \begin{que}
     Let $K_1$ and $K_2$ be two knots. Can one find a fixed positive integer $N$ such that
     \begin{equation*}
         l(K_1\# K_2)\leq l(K_1)+l(K_2)+N
     \end{equation*}
     for every pair of knots $K_1$ and $K_2$?
 \end{que}
 
\bibliographystyle{amsplain}
\bibliography{Ref}

\begin{flushleft}
    Boston College. Massachusetts, USA.\\
    naserisa@bc.edu 
\end{flushleft}

\end{document}